\documentclass{article}

\usepackage[english]{babel}

\usepackage[letterpaper,top=2cm,bottom=2cm,left=3cm,right=3cm,marginparwidth=1.75cm]{geometry}

\usepackage{amsmath}
\usepackage{amsthm}
\usepackage{amsfonts,amssymb}
\usepackage{graphicx}
\usepackage{hyperref}
\hypersetup{colorlinks=true, allcolors=blue}
\usepackage{setspace}
\usepackage{siunitx} 
\usepackage{graphicx,pstricks,listings,stackengine}
\usepackage{enumerate} 
\usepackage{authblk}
\usepackage{float} 
\newtheorem{defn}{Definition}
\newtheorem{thm}{Theorem}
\newtheorem{prop}{Proposition}
\newtheorem{lem}{Lemma}
\newtheorem{cor}{Corollary}
\newtheorem{eg}{Example}
\newtheorem{rmk}{Remark}
\newtheorem{algo}{Algorithm}
\title{On Periodicity of Continued fractions with Partial Quotients in Quadratic Number Fields}
\author{Zhaonan Wang}
\author{Yingpu Deng}
\affil{Key Laboratory of Mathematics Mechanization, NCMIS, Academy of Mathematics and Systems Science, Chinese Academy of Sciences, Beijing 100190, People's Republic of China\authorcr and\authorcr 
	School of Mathematical Sciences, University of Chinese Academy of Sciences,
	Beijing 100049, People’s Republic of China
	\authorcr
	znwang@amss.ac.cn, dengyp@amss.ac.cn}
\date{}
\begin{document}
\maketitle
\bibliographystyle{alpha}
\linespread{1.2}
\begin{abstract}
The properties of continued fractions whose partial quotients belong to a quadratic number field K are distinct from those of classical continued fractions. Unlike classical continued fractions, it is currently impossible to identify elements with periodic continued fraction expansions, akin to Lagrange's theorem. In this paper, we fix a real quadratic field K and take an ultimately periodic continued fraction with partial quotients in $\mathcal{O}_K$, and analyze the growth of $|P_n|$ and $|Q_n|$. We also establish necessary and sufficient conditions for a real quartic irrational to have an ultimately periodic continued fraction with partial quotients in $\mathcal{O}_K$. Additionally, we analyze a specific example with $K=\mathbb{Q}(\sqrt{5})$. By the obtained results, we give a continued fraction expansion algorithm for those real quartic irrationals belong to a quadratic extension of $K$, whose 
algebraic conjugates are all real, and prove that the expansion obtained using our algorithm will be ultimately periodic.
\end{abstract}
\section{Introduction}\label{sec1}
Simple continued fractions have a rich history dating back to the ancient Greeks, and it continues to be widely studied to this day due to their remarkable properties and applications, such as their relationship with quadratic irrationals and Diophantine approximation. In 1770, Lagrange introduced a well-known theorem in \cite{Lagrange}, which states that the simple continued fraction expansion of a real number is ultimately periodic if and only if it is a quadratic irrational. Also, the simple continued fraction of a given positive real number $\xi$ gives the best rational approximation of $\xi$ \cite{Lakein}. Inspired by classical continued fractions, we consider some extended forms of continued fraction expansions and attempt to investigate whether these new forms of continued fractions still possess similar desirable properties. \\
In 1977, Rosen firstly raised a problem in \cite{Rosen}: Can one devise a continued fraction that represents uniquely all real numbers, so that the ﬁnite continued fractions represent the elements of an algebraic number ﬁeld, and conversely, every element of the number ﬁeld is represented by a ﬁnite continued fraction? He defined the $\lambda$-continued fractions for $\lambda=2 \cos(\frac{\pi}{q})$ with partial quotients multiples of $\lambda$. He proved that when $\lambda=\frac{\sqrt{5}+1}{2}$ (i.e. $q$=5), the $\lambda$-continued fraction will satisfy the desired properties.\\
The $\beta$-numeration introduced by \cite{Renyi} and \cite{Parry} generates the integral-base numeration system with the non-integral $\beta$-base. Let $\beta>1$, one can expand any positive real $x $ as $x=\sum\limits_{k=-\infty}^N x_k \beta^{k}$ where $x_k\in\{0,1,\cdots,\lfloor\beta\rfloor\}$. Among all the expansions of $x $, the greatest sequence for the lexicographical order is defined to be the $\beta $-expansion of $x$. The sum consisting of non-negative powers of the expansion of $x $ is called the $\beta $-integral part of $x $, denoted by $[x]_\beta$. Also they defined the (positive) $\beta$-integers by those positive elements $x $ such that $x=[x]_\beta$, denoted by $\mathbb{Z}_\beta^+$.\\
Based on the above definitions, it is natural to consider continued fraction expansions with partial quotients being $\beta$-integers. In 2006, Bernat defined such $\beta$-continued fraction expansion algorithm with $\beta=\frac{\sqrt{5}+1}{2}$ being the golden ratio in \cite{Bernat}, and he showed that any element in $\mathbb{Q}(\beta)$ can be represented in finitely. In \cite{2020}, Zuzana Mas{\'a}kov{\'a} et.al. used a different approach from Bernat to consider more general continued fractions with partial quotients belonging to some discrete subset of the ring of algebraic integers of a real quadratic field $K $, hence Bernat's result is a special case when the discrete subset is chosen to be $\mathbb{Z}_\beta^+$. The authors showed that for all quadratic Perron numbers $\beta$, every element of $\mathbb{Q}(\beta)$ has a periodic $\beta$-continued fraction expansion, and there exist four quadratic Perron numbers including $\beta=\frac{\sqrt{5}+1}{2}$ such that every element of $\mathbb{Q}(\beta)$ has a finite $\beta$-continued fraction expansion. Moreover, they have also proved that for any quadratic real field K, if an element $\alpha\in K$ has a continued fraction expansion with all partial quotients $a_i\in\mathcal{O}_K\cap\mathbb{R}_{\geq1}$ such that $|a_n^\prime|\leq a_n$, where $a_n^\prime$ is the conjugate of $a_n$ in K, then this expansion is finite or ultimately periodic. However, the $\beta$-continued fraction expansion, as defined, fails to preserve the properties, such as the Lagrange theorem. Specifically, the expansion of quartic irrationals in some quadratic extension of $\mathbb{Q}(\beta)$ may not be ultimately periodic.\\
In this paper, we assume that $K$ is an arbitrary quadratic real field with a ring of algebraic integers $\mathcal{O}_K$ and all partial quotients belong to $\mathcal{O}_K$. In contrast to classic continued fractions and $\beta$-continued fractions, we do not impose a constraint on $a_n$ to be greater than 1. It is worth noting that $a_n$ may also take negative values.\\
The paper is structured as follows: Section \ref{sec2} provides an overview of the properties of continued fractions and multiplicative Weil heights. In Section \ref{sec3}, we present some results on convergence of an ultimately periodic continued fraction expansion with partial quotients in $\mathcal{O}_K$, we also establish the necessary and sufficient condition for the periodicity of a continued fraction expansion converging to some element $\xi$, especially when $\xi$ is a real quartic number. Section \ref{sec4} presents an explicit example with $K=\mathbb{Q}(\sqrt{5})$, where we offer an algorithm to ensure that any real quartic irrational in a quadratic extension of $K$ whose conjugates are all real possesses an ultimately periodic continued fraction expansion with partial quotients in $\mathcal{O}_K$.
 
\section{Preliminaries}\label{sec2}
\subsection{Continued fractions}
Let $\{a_n\}$ be a sequence of real numbers, $a_i \neq 0$ for all $i>0$. Define sequences $\{P_n\}$ and $\{Q_n\}$ as the $\mathcal{Q}$-pair associated to $\{a_n\}$ recursively as follows:
$$P_{-1}=1,\ P_0=a_0,\ P_{n+1}=a_{n+1}P_n+P_{n-1},\ (n\geq0)$$
$$Q_{-1}=0,\ Q_0=1,\ Q_{n+1}=a_{n+1}Q_n+Q_{n-1}.\ (n\geq0)$$
Then we have $$[a_0,a_1,\cdots,a_n]:=\dfrac{P_n}{Q_n}=a_0+\dfrac{1}{a_1+\dfrac{1}{\cdots+\dfrac{1}{a_n}}}.$$
If $\{a_n\}$ is an infinite sequence with $a_n>0$ for all $n\geq 1$, then $\dfrac{P_n}{Q_n}$ converges if and only if $\sum\limits_{i=0}^{\infty} a_i$ approaches infinity, which is equivalent to the condition $\lim\limits_{n\to +\infty} Q_n=+\infty$ \cite{khinchin}, and we can write $[a_0,a_1,a_2,\cdots]$ for the limit $\lim\limits_{n\to+\infty} \dfrac{P_n}{Q_n}$.  If the fraction is convergent, denote $\xi:=[a_0,a_1,a_2,\cdots]$, $\{a_n\}$ is called the sequence of partial quotients of $\xi$, and $a_i$ is the $i$-th partial quotients of $\xi$.
Define the complete quotient $\xi_n=[a_n,a_{n+1},\cdots]$, we have
$$\xi=a_0+\dfrac{1}{a_1+\dfrac{1}{\cdots+\dfrac{1}{a_{n}+\dfrac{1}{\xi_{n+1}}}}}=\dfrac{\xi_{n+1} P_{n}+P_{n-1}}{\xi_{n+1} Q_{n}+Q_{n-1}}.$$
Hence we obtain 
$$\xi_{n+1}=-\dfrac{\xi Q_{n-1}-P_{n-1}}{\xi Q_n-P_n},$$
$$\xi-\dfrac{P_n}{Q_n}=\dfrac{\xi_{n+1} P_{n}+P_{n-1}}{\xi_{n+1} P_{n}+P_{n-1}}-\dfrac{P_n}{Q_n}=\dfrac{(-1)^n}{Q_n (\xi_{n+1} Q_n+Q_{n-1})}\ .$$
\begin{prop}\label{prop1}
For all $n\geq0$ we have the following properties:
\begin{enumerate}
    \item $P_{n-1} Q_n-P_n Q_{n-1}=(-1)^n$,
    \item $|\xi-\dfrac{P_n}{Q_n}|=\dfrac{1}{Q_n (\xi_{n+1} Q_n+Q_{n-1})}$,
    \item $|Q_n\xi-P_n|=|\xi_1\cdots\xi_{n+1}|^{-1}$ for all $n\geq0$.
\end{enumerate}
\end{prop}
For proof one can see the first chapter in \cite{khinchin}. Note that if all $a_n\geq1$ for all $n>0$, then $Q_{n+1}>Q_{n}>0$ and $a_{n+1}+1\geq\xi_{n+1}\geq a_{n+1}$ for all $n\geq0$. Specifically, $\dfrac{1}{Q_n\left(Q_{n+1}+Q_n\right)} \leq\left|\xi-\frac{P_n}{Q_n}\right|=\dfrac{1}{Q_n\left(\xi_{n+1} Q_n+Q_{n-1}\right)} \leq \dfrac{1}{Q_n Q_{n+1}} \leq \dfrac{1}{a_{n+1} Q_n^2}$.\\
Now let us examine $P_n$ and $Q_n$ from the perspective of so-called continuants.\\
Define the multivariate polynomials by the recurrence
$$K_{-1}=0,\ K_0=1,\ K_n(t_1,\cdots,t_n)=t_n K_{n-1}(t_1,\cdots,t_{n-1})+K_{n-2}(t_1,\cdots,t_{n-2}).$$ 
It is evident, based on the provided definition, that $K_n$ is a polynomial in $t_1,\cdots,t_n$ featuring positive integer coefficients. Furthermore, each $t_i $ appears in at least one one monomial with a non-zero coefficient.\\
Therefore, we can represent $P_n$ and $Q_n$ as $P_n=K_{n+1}(a_0,\cdots,a_n)$ and $Q_n=K_n(a_1,\cdots,a_n)$. \\
The continuants posses further remarkable properties, and we recommend readers who are interested in learning more to refer to \cite{Graham}.\\
Here we illustrate a proposition that is relevant to this paper,the proof of which can be obtained by induction.
\begin{prop}\label{prop2}
    The polynomials $K_n(t_1,\cdots,t_n)$ defined above satisfy the following properties:
    \begin{enumerate}
        \item $K_n(t_1,\cdots,t_n)=K_n(t_n,\cdots,t_1)$.
        \item For all $j,l\geq1$, we have
$$K_{j+l}\left(t_1, \ldots, t_{j+l}\right)=K_j\left(t_1, \ldots, t_j\right) K_l\left(t_{j+1}, \ldots, t_{j+l}\right)+K_{j-1}\left(t_1, \ldots, t_{j-1}\right) K_{l-1}\left(t_{j+2}, \ldots, t_{j+l}\right).$$
    \end{enumerate}
\end{prop}

\subsection{Weil heights of algebraic numbers}
In this paper, we provide definitions and properties of the Weil height that are pertinent to our research, and we omit the corresponding proofs. Interested readers may refer to the first two chapters of \cite{WeilHight} for a detailed exposition.
\begin{defn}\label{defn1}
Let K be an algebraic number field, let $\mathcal{M}_K=\{\ |\cdot |_v: v\ is\ a\ place\ of\ K\}$ be a collection of normalized absolute values on K (i.e. the Product formula $\prod\limits_{v\in\mathcal{M}_K}|x|_v=1$ holds for all nonzero elements $x\in K$). The (multiplicative) Weil height of x is defined by
$$H(x)=\prod\limits_{v\in\mathcal{M}_K}\max\{|x|_v,1\}.$$
$\rm{(}$ Alternatively one can define the logarithmic Weil height $h(x)=\sum\limits_{v\in\mathcal{M}_K}\log^+|x|_v$, and $H(x)=e^{h(x)}$. $\rm{)}$
\end{defn}
In practice, let $u$ be a place of $\mathbb{Q}$ and $v$ a place of K where $v$ extends $u$ in K. Let $||\cdot||_u$ be the normalized absolute value in $\mathbb{Q}$, we can find the absolute value $||\cdot||_v$ in the place $v$ such that $||\cdot||_v$ equals $||\cdot||_u$ when restricting to $\mathbb{Q}$. Now define another absolute value attached to $v$ for all $v$: $|\cdot|_v=||\cdot||_v^{n_v/n}$, where $n=[K:\mathbb{Q}]$, $n_v=[K_v:\mathbb{Q}_v]$, and $K_v$ (resp. $\mathbb{Q}_v$) is the completion of K (resp. $\mathbb{Q}$) with respect to the place $v$, then $\mathcal{M}_K$ can be defined as the collection of all $|\cdot|_v$.
\begin{rmk}\label{rmk1}
    There are only finitely many factors in the infinite product of $H(x)$ that do not equal 1. Also, the Weil height of $x$ is independent from the field we choose that contains $x$.
\end{rmk}
Next we introduce some useful properties of the Weil height.
\begin{prop}\label{prop3}
Let $ r / s $ be a rational number, $\zeta$ a root of unity, and let $ \alpha \neq 0  $ and $ \beta \neq 0 $ be elements of $ \overline{\mathbb{Q}}^{\times} $. Then
    \begin{enumerate}
\item $H(\alpha\pm \beta)\leq 2 H(\alpha) H(\beta)$,
\item $H(\alpha \beta) \leq H(\alpha) H(\beta)$,
\item $H(\alpha^{r / s})=|r / s|_{\infty} H(\alpha)$, 
\item $H(\sigma(\alpha))=H(\alpha)$ for all $\sigma\in\rm{Gal
}(\overline{\mathbb{Q}}/\mathbb{Q})$,
\item $H(\alpha)=1 $ if and only if $ \alpha$ is a root of unity.
    \end{enumerate}
\end{prop}
\begin{prop}\label{prop4}
$(\mathbf{Northcott^\prime s\ theorem\ on\ heights})$ The set $\{\alpha\in \overline{\mathbb{Q}} \mid H(\alpha) \leq C,\ [\mathbb{Q}(\alpha): \mathbb{Q}] \leq d\}$ is finite for any positive constants $C,d$.
\end{prop}
\begin{prop}\label{prop5}
Let $\alpha\in\overline{\mathbb{Q}}$, $n=[\mathbb{Q}(\alpha):\mathbb{Q}]$ and d be the leading coefficient of the minimal polynomial of $\alpha$ over $\mathbb{Z}$. Then
$$H(\alpha)^n=|d|\prod\limits_{\sigma} \max \{1,|\sigma(\alpha)|\},$$
where $\sigma$ runs through all embeddings of $\mathbb{Q}(\alpha)\hookrightarrow \mathbb{C}$. 
\end{prop}

\section{Continued fractions with partial quotients in quadratic fields}\label{sec3}
In this section, we assume K is a real quadratic number field with the ring of algebraic integers $\mathcal{O}_K$, and 
$[a_0,\cdots,a_N,\overline{a_{N+1},\cdots,a_{N+k}}]$ is an ultimately periodic continued fraction expansion with partial quotients in $\mathcal{O}_K$. 
\subsection{Convergence and periodicity}
We firstly provide some definitions and notations related to formal continued fraction expansions and examine the convergence of this expansion.\\
 For any $\alpha \in \mathbb{C}$, define the matrix
$D(\alpha):=\left[\begin{array}{cc}
\alpha & 1 \\
1 & 0
\end{array}\right] \text {, and set } t:=D(0); \text { note that } D(\alpha)^{-1}=t D(-\alpha) t=\left[\begin{array}{cc}
0 & 1 \\
1 & -\alpha
\end{array}\right] \text {. }
$
Let $F=\left[a_1, a_2, a_3, \ldots, a_n\right]$ be a finite continued fraction. Define
$$
M(F)=\left[\begin{array}{ll}
P_n & P_{n-1} \\
Q_n & Q_{n-1}
\end{array}\right]=\prod_{i=1}^n D\left(a_i\right)=D\left(a_1\right) D\left(a_2\right) \cdots D\left(a_n\right).
$$
It is worth noting that this is precisely the standard matrix employed for representing a continued fraction in matrix form. One can observe that 
$$\begin{aligned} M(F)^{-1} & =\prod_{i=0}^{n-1}\left[t D\left(-a_{n-i}\right) t\right]=t\left[\prod_{i=0}^{n-1} D\left(-a_{n-i}\right)\right] t \\ & =M\left(\left[0,-a_n, \ldots,-a_1, 0\right]\right).\end{aligned}$$
And $\det$ M(F) = $\det M(F)^{-1}=(-1)^n$.\\
Note that the periodic continued fraction $[b_1, \ldots, b_N, \overline{a_1, \ldots, a_k}]$ can be expressed as a purely periodic continued fraction $[\overline{b_1, \ldots, b_N,a_1, \ldots, a_k,0,-b_N, \ldots,-b_1, 0}]$, as the expression $[\cdots a,0,b,\cdots]$ is equivalent to $[\cdots,a+b,\cdots]$, hence for an ultimately periodic continued fraction
$P=[b_1,\cdots,b_N,\overline{a_1,\cdots,a_k}],$ we define 
$$\begin{aligned} 
E(P) & =\left[\begin{array}{ll}
E_{11}(P) & E_{12}(P) \\
E_{21}(P) & E_{22}(P)
\end{array}\right]:=M\left(\left[b_1, \ldots, b_N, a_1, \ldots, a_k, 0,-b_N, \ldots,-b_1, 0\right]\right) \\
& =D\left(b_1\right) \cdots D\left(b_N\right) D\left(a_1\right) \cdots D\left(a_k\right) t D\left(-b_N\right) \cdots D\left(-b_1\right) t \\
& =M\left(\left[b_1, \ldots, b_N\right]\right) M\left(\left[a_1, \ldots, a_k\right]\right) M\left(\left[b_1, \ldots, b_N\right]\right)^{-1},
\end{aligned}$$
It should be noted that the determinant of the matrix $E(P)$ is given by $(-1)^k$, which is solely determined by the periodic part.\\
Finally, we define the polynomial $f(P)=E_{21}(P)x^2+(E_{22}(P)-E_{11}(P))x-E_{12}(P)=0$, and it is called the polynomial associated to the periodic continued fraction expansion $P$. We write $E_{ij}$ instead of $E_{ij}(P)$ when there is no risk of causing misunderstanding.\\ 
 In \cite{s-int}, the authors describe an algorithm to determine if a periodic continued fraction expansion $[a_0,a_1,\cdots,a_n,\overline{a_{n+1},\cdots,a_{n+k}}]$ with partial quotients in rings $\mathcal{O}$ of $S$-integers in a number field is convergent, hence it is applicable to our case where all $a_i\in\mathcal{O}_K$. \\
The algorithm can be illustrated as follows:
\begin{algo}\label{alg1}
The value (or limit) of a periodic continued fraction:\\
$\textbf{Input:}$ A periodic continued fraction $P=[b_1, \ldots, b_N, \overline{a_1, \ldots, a_k}]$.\\
$\textbf{Output:}$ Either its value (or limit) $\hat{\gamma} (P)$ or "doesn't exist".
\begin{enumerate}
    \item compute $E=E(P)$ from the above definition.
    \item if E is a multiple of the identity then print "doesn't exist" and break;
    \item compute the multi-set of roots $\{\gamma ,\gamma^*\}$ of $f(P)=E_{21}x^2+(E_{22}-E_{11})x-E_{12}=0;$
    \item if $\gamma=\gamma^*$, then $\hat{\gamma}(P):=\gamma=\gamma^*$, return $\hat{\gamma}(P)$, break;
    \item Without loss of generality assume $\gamma\neq \infty$ and compute $|E_{21}\gamma+E_{22}|$;
    \item if $|E_{21}\gamma+E_{22}|=1$ then print "doesn't exist"; break;
    \item if $|E_{21}\gamma+E_{22}|>1$ then $\hat{\gamma}(P):=\gamma$, else $\hat{\gamma}(P):=\gamma^*$;
    \item Let $\textbf {INEQ}$ be:$$\ M\left(\left[a_{j+1}, \ldots, a_{k+j}\right]\right)_{21}=0 \text { and }\left|M\left(\left[a_{j+1}, \ldots, a_{k+j}\right]\right)_{22}\right|>1 \text { for some } 0 \leq j \leq k-1.$$
    Test if $\textbf {INEQ}$ is satisfied: if so, then print "doesn't exist"; break; else return $\hat{\gamma}(P)$.
\end{enumerate}
\end{algo}

\begin{lem}\label{lem!}
    Assume $P=[a_0,\cdots,a_N,\overline{a_{N+1},\cdots,a_{N+k}}]$ is an ultimately periodic continued fraction expansion, where all partial quotients $a_i\in\mathcal{O}_K$. If the quadratic equation $f(P)=E_{21}x^2+(E_{22}-E_{11})x-E_{12}=0$ has a positive determinant which is not a square in K, then the expansion is convergent. If the determinant is negative, the expansion does not converge.
\end{lem}
\begin{proof}
    The proof is straightforward and can be derived directly from the steps of the above algorithm. If $\Delta(f(P))$ is positive and not a square in K, then $\{\gamma,\gamma^*\}$ are real numbers that are quadratic over K (hence being quartic irrationals). Since $E_{21},E_{22}\in \mathcal{O}_K $, the absolute value $|E_{21}\gamma+E_{22} |$ must not equal 1. Therefore, we can find the limit of the continued fraction expansion by applying  Algorithm \ref{alg1}.\\
    In the case where the determinant is negative, we will obtain two complex roots $\gamma,\gamma^*$ with identical absolute values. We observe that  $|E_{21}\gamma+E_{22} |=|E_{21}\cdot\dfrac{E_{11}-E_{22}+\sqrt{(E_{22}-E_{11})^2+4E_{12}E_{21}}}{2E_{21}}+E_{22} |$. As the determinant is negative, we can simplify this expression as follows: 
    $$|E_{21}\gamma+E_{22} |^2=|(\dfrac{E_{11}+E_{22}}{2})^2-\dfrac{(E_{22}-E_{11})^2+4E_{12}E_{21}}{4}|=|E_{11}E_{22}-E_{12}E_{21} |=1.$$
     Therefore, the continued fraction expansion does not converge due to  step 6 of the algorithm.
\end{proof}
\begin{lem}\label{lem@}
If the ultimately periodic continued fraction 
 expansion $[a_0,\cdots,a_N,\overline{a_{N+1},\cdots,a_{N+k}}]$ with partial quotients in 
$\mathcal{O}_K$ converges to a real number $\xi$, then $|P_n|$, $|Q_n|$ tends to infinity.
\end{lem}
\begin{proof}
We prove that $|Q_n|\rightarrow\infty$. Since $\dfrac{P_n}{Q_n}$ is convergent, we could also derive that $|P_n|\rightarrow\infty$.\\
     As stated in Proposition \ref{prop1}, we have $P_{n-1} Q_n-P_n Q_{n-1}=(-1)^n$, hence $\dfrac{P_n}{Q_n}-\dfrac{P_{n-1}}{Q_{n-1}}=\dfrac{(-1)^{n-1}}{Q_{n-1}Q_n}$. Therefore, if the continued fraction expansion is convergent, this difference should approach 0, and $|Q_{n-1}Q_n|\rightarrow\infty$. For that reason, there is no infinite subsequence $\{Q_{n_i}\}$ with bounded absolute values, and simultaneously having an infinite number of pairs $(n_{i-1},n_i)$ being adjacent integers.\\
    Since the lemma is in fact equivalent to the statement that there does not exist an infinite subsequence $\{Q_{n_i}\}$ with bounded absolute values, we will prove it by contradiction, assuming that $\{Q_{n_i}\}$ is such a sequence.\\
   Since the continued fraction expansion of $\xi$ is ultimately periodic, we may assume that the value-bounded subsequence $\{Q_{n_i}\}$ satisfies the condition that $a_{n_i}$ are identical for all indices $n_i$ appearing in the subsequence, otherwise we can extract a proper subsequence from $\{Q_{n_i}\}$ to meet the condition. Based on the proceeding discussion, it can be observed that for the satisfactory subsequence $\{Q_{n_i})\}$, the absolute values of the subsequence $\{Q_{n_i-1}\}$ will tend towards infinity, hence the absolute values of $\{\dfrac{Q_{n_i}}{Q_{n_i-1}}\}$ approaches 0.\\
   We have $Q_{n_i}=a_{n_i} Q_{n_i-1}+P_{n_i-2}$, thus the ratio $\dfrac{Q_{n_i}}{Q_{n_i-1}}$ can be formally expressed as
   $$\dfrac{Q_{n_i}}{Q_{n_i-1}}=a_{n_i}+\dfrac{1}{a_{n_i-1}+\dfrac{1}{\cdots+\dfrac{1}{a_1}}}.$$
   We denote this continued fraction as $[a_{n_i},\cdots,a_1]$, which converges to 0 as $n_i$ increases.\\
   It's noteworthy that if $[\overline{\alpha_0,\cdots,\alpha_m}]$ is a purely periodic continued fraction expansion taking a real root of the polynomial $ax^2+bx+c=0$ as the limit, then $[\overline{\alpha_m,\cdots,\alpha_0}]$ also converges: The key is to verify that this expansion meets all the convergence criteria outlined in the algorithm above. In fact, assuming that the final two convergents of $[\alpha_0,\cdots,\alpha_m]$ are $\dfrac{p_{m-1}}{q_{m-1}}$ and $\dfrac{p_m}{q_m}$, and those of $[\alpha_m,\cdots,\alpha_0]$ are $\dfrac{p^\prime_{m-1}}{q^\prime_{m-1}}$ and $\dfrac{p^\prime_m}{q^\prime_m}$, we can deduce from Proposition \ref{prop2} that 
   $$p^\prime_m=p_m,\ q^\prime_m=p_{m-1},$$
   $$p^\prime_{m-1}=q_m,\ q^\prime_{m-1}=q_{m-1}. $$
   We have $a=q_m,\ b=-(p_m-q_{m-1}),\ c=-p_{m-1}$, thus the quadratic equation corresponding to  $[\overline{\alpha_m,\cdots,\alpha_0}]$ is $cx^2-bx+a=0$, which has the same discriminant as $ax^2+bx+c=0$. Therefore, according to Lemma \ref{lem!}, $[\overline{\alpha_m,\cdots,\alpha_0}]$ is convergent to one root of $cx^2-bx+a=0$.\\
   Based on our previous assumptions, we may assume the periodic part of the expansion of $\xi$ ends with the $a_{n_i}$ chosen above. i.e., if we denote $a_{n_i}:=b_k$, then $\xi$ can be expressed as $[a_0,\cdots,a_{N^\prime},\overline{b_1,\cdots,b_k}]$ for some integer $N^\prime$ with $(b_1,\cdots,b_k)$ obtained by appropriately cyclically shifting $(a_{N+1},\cdots,a_{N+k} )$ to ensure the following holds:
   $$\dfrac{Q_{n_i}}{Q_{n_i-1}}=[b_k,\cdots,b_1,b_k,\cdots,b_1,\cdots,a_{N^\prime},\cdots,a_1].$$
   Denote $\alpha:=[a_{N^\prime},\cdots,a_1]$,  $l_i:=\dfrac{n_i-N^\prime}{k}$ (by above discussion $l_i\in\mathbb{Z}_+$), and the last two convergents of $[\underbrace{b_k,\cdots,b_1}_{1},\underbrace{b_k,\cdots,b_1}_{2},\cdots,\underbrace{b_k,\cdots,b_1}_{l_i}]$ are $\dfrac{\gamma_{i,1}}{\delta_{i,1}}$ and $\dfrac{\gamma_{i,2}}{\delta_{i,2}}$, both of which converges to $[\overline{b_k,\cdots,b_1}]$. Then we have $\dfrac{Q_{n_i}}{Q_{n_i-1}}=\dfrac{\alpha \gamma_{i,1}+\gamma_{i,2}}{\alpha \delta_{i,1}+\delta_{i,2}}$, hence $\lim\limits_{n_i\rightarrow\infty}\dfrac{Q_{n_i}}{Q_{n_i-1}}=[\overline{b_k,\cdots,b_1}]$, which cannot be zero. This leads to a contradiction to the initial assumption.
\end{proof}
By means of a proof similar to that of the theorem, we can establish the following corollary:
\begin{cor}\label{cor!}
    If the ultimately periodic continued fraction 
 expansion $[a_0,\cdots,a_N,\overline{a_{N+1},\cdots,a_{N+k}}]$ with partial quotients in 
$\mathcal{O}_K$ converges to some real number $\xi$, then there does not exist a subsequence $\{Q_{n_i}\}$ or $\{P_{n_i}\}$ such that $|\dfrac{Q_{n_i}}{Q_{n_i-1}}|$ or $|\dfrac{P_{n_i}}{P_{n_i-1}}|$ tend to infinity or 0. Therefore, $|\dfrac{Q_n}{Q_{n-1}}|$ and $|\dfrac{P_n}{P_{n-1}}|$ are bounded and have a lower bound larger than 0.
\end{cor}

Now we consider a continued fraction expansion $P=[a_0,\cdots,a_N,\overline{a_{N+1},\cdots,a_{N+k}}]$, which is ultimately periodic with all partial quotients $a_i\in\mathcal{O}_K$. Assume the discriminant of $f(P)$ is positive, then by Lemma \ref{lem!}, it converges to some real number $\xi_P$. We may assume $\Delta(f(P))$ is not a real square in K, i.e. $\xi_P$ is a quartic irrational, otherwise $\xi_P$ will be a quadratic irrational, and by the well-known Lagrange theorem it can be expanded into an ultimately periodic continued fraction with $a_i\in\mathbb{Z}_+$ for all $i\geq1$.  We call it a continued fraction expansion of $\xi_P $ in K. Notice that all complete quotients $\xi_n$ lie in the same quadratic extension L over K, where $L=K(\sqrt{\Delta(f(P))}$. Hence this is a necessary condition for an element $\xi$ to admit a periodic expansion over K. Next we assume $\xi$ is a real quartic irrational, lying in a quadratic extension of K.\\
Denote $\rm{Gal}(K/\mathbb{Q})=\{1,\sigma\}$, where for all $\alpha\in K$, $\sigma$ maps $\alpha$ to its (real) conjugate $\overline{\alpha}$. Assume a real quartic irrational $\xi$ has minimal polynomial $f(x)=Ax^2+Bx+C$ over K, where $A,B,C\in\mathcal{O}_K$,  and the discriminant $B^2-4 A C > 0$ is not a square in $\mathcal{O}_K$. We can describe the periodicity of a continued fraction expansion alternatively through the finite existence of the complete quotients $\xi_n$. Since all $\xi_n\in L\setminus K$, where $L=K(\sqrt{B^2-4AC})$, we can apply the Northcott theorem in Proposition \ref{prop4}. Therefore, we can deduce that if a continued fraction expansion $[a_0,\cdots,a_n,\cdots]$ converges to $\xi$, then it is ultimately periodic  if and only if there exists an upper bound for all $H(\xi_{n+1})$. Here for convenience we use  the shifted index $n+1$.\\
With the formula mentioned in Proposition \ref{prop5} we have 
$$H(\xi_{n+1})^4=d \prod\limits_{\phi:\rm{embeddings\ of\ }L\hookrightarrow \mathbb{C}} \max \{1,|\phi(\xi_{n+1})|\},$$
where $d$ is the leading coefficient of the minimal polynomial of $\xi_{n+1}$ over $\mathbb{Z}$.\\
Now we calculate the minimal polynomial $f_{n+1}$ of all complete quotients $\xi_{n+1}$ over $\mathcal{O}_K$.
It is worth noting that the calculation is analogous to that in the classical continued fraction case, which can be found in \cite{khinchin}. Hence we have
$$f_{n+1}(x):=A_{n+1} x^2+ B_{n+1} x + C_{n+1}=0,$$
where
$$\left\{\begin{array}{l}
			A_{n+1}=A P_n^2+B P_n Q_n +C Q_n^2, \\
			B_{n+1}=2 A P_n P_{n-1}+B(P_n Q_{n-1}+P_{n-1} Q_n)+2 C Q_n Q_{n-1}, \\
			C_{n+1}=A P_{n-1}^2+B P_{n-1} Q_{n-1} +C Q_{n-1}^2=A_n.
		\end{array}\right.$$
  
  Then the minimal polynomial of $\xi_{n+1}$ over $\mathbb{Z}$ is $f_{n+1}(x)\cdot \sigma(f_{n+1}(x))$, and the leading coefficient of the minimal polynomial of $\xi_n$ over $\mathbb{Z}$ is $A_{n+1}\cdot \sigma(A_{n+1})$. Then we can proceed to compute the Weil height of $\xi_{n+1}$:
	$$H(\xi_{n+1})^4=|A_{n+1}\cdot\sigma(A_{n+1})|\prod\limits_{\phi} \max\{1,|\phi(\xi_{n+1})|\}=|A_{n+1}\cdot\sigma(A_{n+1})|\prod\limits_{\phi} \max\{1,|\phi(\dfrac{\xi Q_{n-1}-P_{n-1}}{\xi Q_n-P_n})|\}.$$
	Multiply then divide the term $\prod\limits_{\phi}|\phi(\xi Q_n-P_n)|$, we obtain $$H(\xi_{n+1})^4=\dfrac{|A_{n+1}\cdot\sigma(A_{n+1})|}{\prod\limits_{\phi}|\phi(\xi Q_n-P_n)|}\prod_{\phi}\max \{|\phi(\xi Q_{n-1}-P_{n-1})|,|\phi(\xi Q_n-P_n)|\}.$$
 
 We define $S_n=\xi Q_n-P_n$. Since $\xi$ satisfies the quadratic equation $Ax^2+Bx+C=0$ over $\mathcal{O}_K$, $S_n$ is the root of $A(\dfrac{x+P_n}{Q_n})^2+B(\dfrac{x+P_n}{Q_n})+C=0$. Simplifying this expression yields $Ax^2+(BQ_n+2AP_n)x+A_{n+1}=0$, thus we have $\prod_{\phi} |\phi(S_n)|=\dfrac{|A_{n+1}\cdot\sigma(A_{n+1})|}{|A\cdot\sigma(A)|}$, and $H(\xi_{n+1})^4=|A\cdot\sigma(A)|\prod_{\phi}\max\{|\phi(S_n)|,|\phi(S_{n-1})|\}$.\\ 
Based on the previous discussion, we can explicitly express all $\phi$. Since $\phi$ maps $\xi$ to one of the roots of the equation $(Ax^2+Bx+C)\cdot(\sigma(A)x^2+\sigma(B)x+\sigma(C))=0$, we denote the maps
$$\tau_1:\sqrt{B^2-4AC}\rightarrow-\sqrt{B^2-4AC},$$
 $$\tau_2:\sqrt{B^2-4AC}\rightarrow\sqrt{\sigma(B^2-4AC)},$$
 $$\tau_3:\sqrt{B^2-4AC}\rightarrow -\sqrt{\sigma(B^2-4AC)},$$
 Then the embeddings of $L\hookrightarrow\mathbb{C}$ are $\{1,\tau_1,\sigma\tau_2,\sigma\tau_3\}$.\\
We give the following lemma to show that both $\prod\limits_{\phi\in\{1,\tau_1\}}\max\{|\phi(S_n)|,|\phi(S_{n-1})|\}$ and $\prod\limits_{\phi\in\{\tau_2\sigma,\tau_3\sigma\}}\max\{|\phi(S_n)|,|\phi(S_{n-1})|\}$ should be bounded for an ultimately periodic expansion.
\begin{lem}\label{lembound}
    If an ultimately periodic continued fraction expansion $P=[a_0,\cdots,a_N,\overline{a_{N+1},\cdots,a_{N+k}}]$ converges to a real quartic irrational $\xi$ with all partial quotients $a_i\in\mathcal{O}_K$, then  $\prod\limits_{\phi\in G_1}\max\{|\phi(S_n)|,|\phi(S_{n-1})|\}$ and $\prod\limits_{\phi\in G_2}\max\{|\phi(S_n)|,|\phi(S_{n-1})|\}$ are bounded for all $n$, where $G_1=\{1,\tau_1\}$, $G_2=\{\tau_2\sigma,\tau_3\sigma\}$.
\end{lem}
\begin{proof}
    For convenience we denote $F_i(n)=\prod\limits_{\phi\in G_i}\max\{|\phi(S_n)|,|\phi(S_{n-1})|\}$. Based on the assumption of periodicity, it is evident that  $F_1(n)\cdot F_2(n)$ are bounded for all $n$. Therefore, to prove the first part of the theorem, it suffices to demonstrate that $F_1(n)$ is upper and lower bounded for all $n$. \\
    $F_1(n)=\max\{|\xi Q_n-P_n|,|\xi Q_{n-1}-P_{n-1}|\}\cdot\max\{|\tau_1(\xi)Q_n-P_n|,|\tau_1(\xi)Q_{n-1}-P_{n-1} |\}$. Again by Proposition \ref{prop1} we have $|\xi Q_n-P_n|=\frac{1}{|\xi_{n+1}Q_n+Q_{n-1}|}$, hence $|Q_n|\cdot|\xi Q_n-P_n|=\frac{1}{|\xi_{n+1}+\frac{Q_{n-1}}{Q_{n}}|}$. Next we show that $\frac{1}{|\xi_{n+1}+\frac{Q_{n-1}}{Q_{n}}|}$ is bounded if the continued fraction expansion is periodic, hence $|\xi Q_n-P_n|\to0$.\\
    As stated in Corollary \ref{cor!}, it is lower bounded, hence it suffices to show that there is no subsequence of $\{|\xi_{n+1}+\frac{Q_{n-1}}{Q_{n}}|\}$ that tends to 0. We still prove this by contradiction. Suppose there exists such a subsequence $\{|\xi_{n_i+1}+\frac{Q_{n_i-1}}{Q_{n_i}}|\}$, we may assume that all $a_{n_i}$ are identical, and the period length of the expansion is $k$. From the proof of lemma \ref{lem@}, we have $\lim\limits_{n_i\to\infty} \frac{Q_{n_i}}{Q_{n_i-1}}=[\overline{a_{n_i},a_{n_i-1},\cdots,a_{n_i-k+1}}]$, and $\xi_{n_i+1}=[\overline{a_{n_i+1},\cdots,a_{n_i+k}}]$. Assume the limit of  $[\overline{a_{n_i},a_{n_i-1},\cdots,a_{n_i-k+1}}]$ is $\gamma$, we shall prove that $-\frac{1}{\gamma}\neq\xi_{n_i+1}$.\\
    Let $P_1=[\overline{a_{n_i+1},\cdots,a_{n_i+k}}]=[\overline{a_{n_i+k-1},\cdots,a_{n_i}}]$, $P_2=[\overline{a_{n_i},a_{n_i-1},\cdots,a_{n_i-k+1}}]$. Since $P_1$ and $P_2$ are purely periodic, we have 
$E(P_i)=M(P_i)$. According to the discussion in the proof of Lemma \ref{lem@}, if we assume  $E(P_1)=M(P_1)=\left[\begin{array}{ll}
E_{11} & E_{12} \\
E_{21} & E_{22}
\end{array}\right]$, then $E(P_2)=M(P_2)=E(P_1)^T=\left[\begin{array}{ll}
E_{11} & E_{21} \\
E_{12} & E_{22}
\end{array}\right]$. Since $\xi_{n_i+1}$ is the limit of $P_1$, it is a root of $f(P_1)$ which satisfies $|E_{21}\xi_{n_i+1}+E_{22}|>1$, by Algorithm \ref{alg1}. Denote the other root of $f(P_1)$ by $\xi_{n_i+1}^*$, hence $|E_{21}\xi_{n_i+1}^*+E_{22}|<1$. Without loss of generality, we may assume $\xi_{n_i+1}=\frac{E_{11}-E_{22}+\sqrt{(E_{11}-E_{22})^2+4E_{21}E_{12}}}{2E_{21}}$; the other case can be proven in a similar way. Then $-\frac{1}{\xi_{n_i+1}}=\frac{E_{11}-E_{22}-\sqrt{(E_{11}-E_{22})^2+4E_{21}E_{12}}}{2E_{12}}$, and $|-\frac{1}{\xi_{n_i+1}}\cdot E_{12}+E_{22}|=|\frac{E_11+E_{22}+\sqrt{(E_{11}-E_{22})^2+4E_{21}E_{12}}}{2}|=|\xi_{n_i+1}^*E_{21}+E_{22}|<1$. Therefore, $\gamma\neq-\frac{1}{\xi_{n_i+1}}$, since $|\gamma E_{12}+E_{22}|>1$.\\
This claim allows us to assume $\frac{1}{|\xi_{n+1}+\frac{Q_{n-1}}{Q_n}|}\in[t_1,t_2]$, and $|\frac{Q_n}{Q_{n-1}}|\in[t_3,t_4]$ by Corollary \ref{cor!}, where both $t_1,t_3\geq0$. Therefore, 
for any $\epsilon>0$, there exists an $n_\epsilon$ such that for all $n\geq n_\epsilon$, we have
$$\begin{aligned}
    F_1(n) & \leq \max\{|\xi Q_n-P_n|,|\xi Q_{n-1}-P_{n-1}|\}\cdot\max\{|Q_n|,|Q_{n-1}|\}\cdot\max\{|\tau_1(\xi)-\frac{P_n}{Q_n}|,|\tau_1(\xi)-\frac{P_{n-1}}{Q_{n-1}}|\}\\
    & \leq t_2 t_4(|\tau_1(\xi)-\xi|+\epsilon),
\end{aligned}
$$
and similarly
$$F_1(n)\geq t_1 t_3 (|\tau_1(\xi)-\xi|-\epsilon).$$
Therefore, we can choose $\epsilon$ to be sufficiently small, such that $|F_1(n)|$ is bounded and has a lower bound greater than zero, thus proving the lemma.
\end{proof}
The following theorem will give a sufficient and necessary condition for the periodicity of a continued fraction expansion over K.
\begin{thm}\label{thmbound}
Assume a continued fraction expansion $[a_0,\cdots,a_n,\cdots]$ with all partial quotients in $\mathcal{O}_K$ converges to a real quartic irrational $\xi$, which is quadratic over K, and the minimal polynomial of $\xi$ over $K$ is $Ax^2+Bx+C=0$, then it is ultimately periodic if and only if $|A_n|$ and $|\sigma(A_n)|$ are bounded for all $n$, where $A_n=A P_n^2+B P_n Q_n+C Q_n^2$ as defined before.
\end{thm}
\begin{proof}
    The necessity of the proof is straightforward, since $F_1(n)\geq|(\xi Q_n-P_n)\cdot(\tau_1(\xi) Q_n-P_n)|=|A_n|$, and $F_2(n)\geq|(\tau_2\sigma(\xi) \sigma(Q_n)-\sigma(P_n))\cdot(\tau_3\sigma(\xi) \sigma(Q_n)-\sigma(P_n))|=|\sigma(A_n)|$, and $|F_i(n)|$ are bounded by the Lemma \ref{lembound}. \\
    To prove the sufficiency, we recall the definition of naive height. Let $\alpha$ be an algebraic number of degree $d$, if the minimal polynomial of $\alpha$ has integer coefficients $c_0, \ldots, c_d$, then the naive height of $\alpha$ is defined as 
$$
h(\alpha)=\max \left\{\left|c_i\right|: i=0, \ldots, d\right\}.
$$
The connection between naive height and (multiplicative) Weil height concerning algebraic integers is well-established. Specifically, the naive height and (multiplicative) Weil height are equivalent for $\alpha\in\mathbb{Q}$. 
 For $\alpha\in\overline{\mathbb{Q}}$ with degree $d$,  it can be observed that: $$ H(\alpha) \leq (d+1)^{1 / 2} h(\alpha) \quad \mathrm{and} \quad h(\alpha) \leq 2^d H(\alpha). $$ 
 Details can be found in \cite{diophantine}.
As mentioned previously, periodicity is equivalent to the finite existence of $\xi_n$, which was originally established by the bound on the Weil height. The aforementioned observation implies that we can alternatively prove this finiteness with the bound on the naive height.\\
Recall that all $\xi_n$ are quartic irrationals with the minimal polynomial over $\mathbb{Z}$
$$f_{n}(x)\cdot\sigma(f_{n})(x)=0,$$
where $$
f_{n}(x)=A_{n} x^2+B_{n} x+C_{n},$$
and $C_{n}=A_{n-1}$ 
(Here $\sigma(f)$ denotes the action of $\sigma$ on every coefficient of the polynomial $f$). The discriminant of $f_{n}$ is equal to $B_{n}^2-4A_{n}C_{n}=\left(P_{n-1} Q_{n-2}-Q_{n-1} P_{n-2}\right)^2(B^2-4AC)=B^2-4AC$, hence the discriminant of $\sigma(f_{n})$ is $\sigma(B^2-4AC) $, both of which remain constant while verifying $n$. Hence assuming $|A_n|$ and $|\sigma(A_{n})|$ are bounded, we can obtain $B_{n}^2\leq(B^2-4AC)+4|A_{n-1} A_n|$ and $\sigma(B_n)^2\leq |\sigma(B^2-4AC)|+4|\sigma(A_{n-1}) \sigma(A_n)|$. As a result, both polynomials $f_{n}(x)$ and $\sigma(f_{n})(x)$ has bounded absolute values of the coefficients, which suggests that the absolute values of the coefficients of the polynomial product are bounded.
\end{proof}

\begin{cor}\label{cor<0}
    With all the assumptions in Theorem \ref{thmbound}, if $\sigma(B^2-4AC)<0$ and  the continued fraction expansion is  ultimately periodic, then $|\sigma(P_n)|$ and $|\sigma(Q_n)|$ are bounded for all $n$.
\end{cor}
\begin{proof}
If $\sigma(B^2-4AC)<0$, then $\sigma(A_{n+1})=\sigma(A)[\sigma(P_n-\frac{B}{2A}Q_n)^2+\sigma(AC-\frac{B^2}{4})\cdot\sigma(Q_n)^2]$, and the absolute value is bounded if and only if $|\sigma(Q_n)|,\ |\sigma(P_n-\frac{B}{2A}Q_n))|$ are bounded, hence if the expansion is ultimately periodic, both $|\sigma(P_n)|$ and $|\sigma(Q_n)|$ are bounded.
\end{proof}

\begin{cor}\label{correstrictions<0}
   With all the assumptions in Theorem \ref{cor<0}, and denote $\xi=[a_1,\cdots,a_N,\overline{a_{N+1},\cdots,a_{N+k}}]$, then the period length $k$ is even, and $|\sigma(B^2-4AC)|<4$. Conversely speaking, if some $\xi$ has minimal polynomial $Ax^2+Bx+C=0$ over $\mathcal{O}_K$ with $\sigma(B^2-4AC)<-4$, then $\xi$ does not have any periodic continued fraction expansions in K.
\end{cor}
\begin{proof}
   Denote this periodic expansion converging to $\xi$ by $P$. From Algorithm \ref{alg1},  $\xi$ is the root of $f(P)=E_{21}(P)x^2+(E_{22}(P)-E_{11}(P))x-E_{12}(P)=0$, and the  discriminant of $f(P)$ is $\Delta(f(P))=(E_{22}(P)-E_{11}(P))^2+4E_{21}(P)E_{12}(P)=(E_{22}(P)+E_{11}(P))^2+4(-1)^{k-1} $. Therefore, if $k$ is odd, then $\sigma(\Delta(f(P)))=\sigma(B^2-4AC)>0$, contradicting the assumption. Furthermore, if $\sigma(B^2-4AC)<-4$, then $B^2-4AC$ cannot coincide with any $\Delta(f(P))$.
\end{proof}
Next we give a necessary and sufficient condition for $\xi$ to have an ultimately periodic continued fraction expansion in K, where the algebraic conjugates of $\xi$ are all real. This theorem is a restatement of Theorem \ref{thmbound}, which reveals to some extent the property of continued fractions as a means of approximating real numbers. 
\begin{thm}\label{thm>0}
With all the assumptions in Theorem \ref{thmbound}, if $\sigma(B^2-4AC)>0$ and the continued fraction expansion is ultimately periodic, then $|\xi Q_n-P_n|$ and $|\xi^\prime\sigma(Q_n)-\sigma(P_n)|$ are bounded. Moreover, there exist positive constants $M_1,M_2$, and a (real) root $\xi^\prime$ of the equation $\sigma(A)x^2+\sigma(B)x+\sigma(C)=0$, such that 
$$|\xi Q_n-P_n|\leq \dfrac{M_1}{|Q_n|},$$
$$|\xi^\prime \sigma(Q_n)-\sigma(P_n)|\leq\dfrac{M_2}{|\sigma(Q_n)|}.$$
The two conditions is also sufficient for a continued fraction expansion of such a $\xi$ to be ultimately periodic.
\end{thm}
\begin{rmk}\label{rmk>0}
In fact, we can obtain the stronger result that $|Q_n|$, $|\sigma(Q_n)|$ tend to infinity by Lemma \ref{lem@}, assuming periodicity. However, but it suffices to use the boundedness of $|\xi Q_n-P_n|$ and $|\xi^\prime\sigma(Q_n)-\sigma(P_n)|$ to prove that the continued fraction expansion is periodic.
\end{rmk}
\begin{proof}
   Denote the ultimately periodic continued fraction $[a_0,a_1,\cdots]$ which converges to $\xi$ by $P$, then $\sigma(P)=[\sigma(a_0),\sigma(a_1),\cdots]$ is also an ultimately periodic continued fraction. Since $\Delta(f(P))=B^2-4AC$, it follows that $\Delta(f\sigma(P)))=\Delta(\sigma(f(P)))=\sigma(B^2-4AC)>0$. Lemma \ref{lem!} implies that  it converges to a root of $f(\sigma(P))=\sigma(A)x^2+\sigma(B)x+\sigma(C)=0$. According to the proof of Lemma \ref{lembound}, it is evident that for any periodic continued fraction converging to some $\alpha$ with $\mathcal{Q}$-pairs $p_n,q_n$, $|q_n|\cdot|\alpha q_n-p_n|=\frac{1}{|\alpha_{n+1}+\frac{q_{n-1}}{q_n}|}$ is bounded, hence it holds for the expansion $P$ and $\sigma(P)$.\\
   To prove the sufficiency, we firstly prove the boundedness of $|A_n|$.
   Now we have $|\xi Q_n-P_n|\leq\frac{M_1}{|Q_n|}$, hence we may assume
   $$P_n=\xi Q_n+\dfrac{M_1 \delta_n}{Q_n}\quad(|\delta_n|<1),$$
and we have
$$\begin{aligned}
    A_{n+1} & =A(\xi Q_n+\dfrac{M_1 \delta_n}{Q_n})^2+B\cdot(\xi Q_n+\dfrac{M_1\delta_n}{Q_n})\cdot Q_n+C\cdot Q_n^2\\
    & = (A\xi^2+B\xi+C)\cdot Q_n^2+(2A\xi+B)\cdot M_1\delta_n+A(\dfrac{M_1 \delta_n}{Q_n})^2.
\end{aligned}
$$
Since $\xi$ is the root of $Ax^2+Bx+C=0$, it follows that 
$$A_{n+1}=(2A\xi+B)\cdot M_1\delta_n+A(\dfrac{M_1 \delta_n}{Q_n})^2.$$
We examine the two cases where $|Q_n|<1$ and $|Q_n|\geq1$ separately. For $|Q_n|\geq 1$, we have
$$|A_{n+1} |\leq(|2A\xi|+|B|)\cdot M_1+|A|M_1^2.$$
On the other hand when $|Q_n|<1$, we observe that assuming $|\xi Q_n-P_n|$ is bounded, then so is $|P_n|$. As a result, $|A_{n+1}|$ is bounded in both cases. Since we can prove that $|\sigma(A_{n+1})|$ is bounded in the similar way, the condition stated in Theorem \ref{thmbound} is satisfied and the continued fraction expansion becomes ultimately periodic.
\end{proof}
\subsection{Constraints on the partial quotients}
In this subsection we establish some conclusions about periodicity by imposing additional constraints on some of the quotients.\\
As demonstrated in previous discussions, the indeterminate sign of the partial quotients  brings significant uncertainty to the analysis of periodicity. If we impose the condition that all $a_n\in\mathcal{O}_K\cap\mathbb{R}_{\geq1}$, as in some existing algorithms (e.g. $\beta$-continued fractions for some $\beta>1$), then from the remark of Proposition \ref{prop1} we know $\left|S_{n-1}|=|\xi Q_{n-1}-P_{n-1}\right|>\left|\xi Q_n-P_n \right|=|S_n|$, and $\frac{1}{Q_n+Q_{n+1}} \leq\left|S_n\right| \leq \frac{1}{Q_{n+1}}$. Therefore, $\frac{1}{2 Q_n} \leq \frac{1}{Q_n+Q_{n-1}} \leq\left|S_{n-1}\right| \leq \frac{1}{Q_n}$.
For $\phi=\tau_1$, it can be seen that 
$\left|\phi\left(S_n\right)\right|=\left|S_n+\left(\tau_1(\xi)-\xi\right) Q_n\right|$. Since $Q_n \rightarrow \infty$ and $S_n \rightarrow 0$, when $n$ is sufficiently large, we have $\left|\tau_1\left(S_n\right)\right| \sim\left(\tau_1(\xi)-\xi\right) Q_n$. Thus there exists an integer $N>0$ such that for all $n \geq N$, $\left|\tau_1\left(S_n\right)\right|>\left|\tau_1\left(S_{n-1}\right)\right|,\left(\tau_1(\xi)-\xi\right) Q_n-\frac{1}{Q_n} \leq\left|\tau_1\left(S_n\right)\right| \leq\left(\tau_1(\xi)-\xi\right) Q_n+\frac{1}{Q_n}$
Now for $n \geq N$ with $N$ chosen as above, we observe that $F_1(n)=\prod_{\phi \in\left\{1, \tau_1\right\}} \max \left\{\left|\phi\left(S_n\right)\right|,\left|\phi\left(S_{n-1}\right)\right|\right\}=$ $\left|S_{n-1}\right| \cdot\left|\tau_1\left(S_n\right)\right| \in\left[\frac{\tau_1(\xi)-\xi}{2}-\frac{1}{Q_n^2}, \tau_1(\xi)-\xi+\frac{1}{Q_n^2}\right]$, which has both upper and lower bounds. Consequently, the periodicity of the expansion is equivalent to the boundedness of $F_2(n)$.\\
\begin{thm}\label{thmconstraints}
Assume that a continued fraction expansion $[a_0,a_1,\cdots]$ converges to a real quartic irrational $\xi$, where all partial quotients $a_n$ are chosen such that $a_n\geq1, \sigma(a_n)>0$ for all $n\geq1$, then:
\begin{itemize}
    \item If all algebraic conjugates of $\xi$ are real numbers, then $\xi$ has an ultimately periodic continued fraction expansion in K if and only if there exists a lower bound for all $\sigma(a_n)$ when $n\geq1$ and $\frac{\sigma(P_n)}{\sigma(Q_n)}$ converges to one of the roots of the equation $\sigma(A)x^2+\sigma(B)x+\sigma(C)=0$. 
    \item If $\sigma(B^2-4AC)<0$, then $\xi$ does not have periodic continued fraction expansion in K.
\end{itemize}
\end{thm}
\begin{proof}
   The proof of this theorem is straightforward. As for the first claim, the ``if'' part follows naturally from Theorem \ref{thm>0}. Now we assume that there exists $\alpha>0$ such that $\sigma(a_n)\geq\alpha>0$ for all $n\geq1$, and $\frac{\sigma(P_n)}{\sigma(Q_n)}$ converges to one of the roots of $\sigma(A)x^2+\sigma(B)x+\sigma(C)=0$, denoted by $\xi^\prime$.
By Proposition \ref{prop1} we have $|\xi^\prime \sigma(Q_n)-\sigma(P_n)|=\dfrac{1}{\xi_{n+1}^\prime \sigma(Q_n)+\sigma(Q_{n-1})}\leq \dfrac{1}{\sigma(Q_{n+1})}\leq\dfrac{1}{\sigma(a_{n+1})\sigma(Q_n)}$. Therefore, we can set $M_2=\dfrac{1}{\alpha}$ to demonstrate the boundedness of $F_2(n)$ and apply Theorem \ref{thm>0} again.\\
As for the second claim, Khinchin in \cite{khinchin} illustrated that, ultimately periodic expansions $[a_0,a_1,\cdots]$ with $\sigma(a_n)>0$ will lead to the divergence of $\sum_{n=1}^{\infty} \sigma\left(a_n\right)$. Consequently, a subsequence in $\sigma\left(Q_n\right)$ with values approaching infinity will be obtained, which contradicts Corollary \ref{cor<0}.
\end{proof}
\begin{cor}\label{cor1}
Assume that the partial quotients are chosen such that $\sigma(a_n)>0$ for all $n\geq1$. If no roots of the equation $\sigma(A)x^2+\sigma(B)x+\sigma(C)=0$ fall within the interval $(\sigma(a_0),\sigma(a_0)+\frac{1}{\sigma(a_1)})$, then the continued fraction expansion of $\xi$ over K starting with $[ \sigma(a_0), \sigma(a_1)]$ will not be ultimately periodic.
\end{cor}
\begin{proof}
  Given that all $\sigma(a_i)$ are positive, we can conclude that the limit of any convergent continued fraction expansion, $[\sigma(a_0),\sigma(a_1),\cdots]$, falls within the interval $(\sigma(a_0),\sigma(a_0)+\frac{1}{\sigma(a_1)})$. Hence the statement can be proved directly using Theorem \ref{thmconstraints}.
\end{proof}
\section{K=$\mathbb{Q}(\sqrt{5})$: An Explicit Example}\label{sec4}
In this section, we take $\beta=\frac{\sqrt{5}+1}{2}$, $K=\mathbb{Q}(\beta)=\mathbb{Q}(\sqrt{5}) $, then $\mathcal{O}_K=\mathbb{Z}[\beta] $. The Galois group $\rm{Gal}(K/\mathbb{Q} )=\{1,\sigma\}$, where $\sigma$ maps $\beta$ to its algebraic conjugate $-\dfrac{1}{\beta}$. We will present a continued fraction 
 expansion algorithm such that for a given element $\xi$, the algorithm generates a continued fraction expansion with partial quotients in $\mathcal{O}_K$ that converges to $\xi$, and it is applicable to all real quartic irrationals $\xi$ which are quadratic over K and have real algebraic conjugates. Moreover, this continued fraction expansion of $\xi$ will be ultimately periodic.\\
 \begin{figure}[H]
\centering
\includegraphics[width=0.45\textwidth]{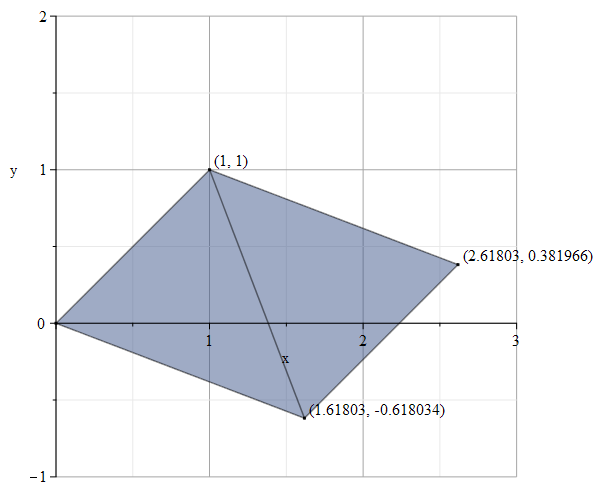}
\caption{The fundamental parallelepiped of $v(\mathcal{O}_K)$}
\label{lattice}
\end{figure}
Now we consider the function $v:K\rightarrow \mathbb{R}^2 $ defined by 
$$v(x)=(x,\sigma(x)),$$
it is clear that $v$ is an additive map.
Recall that an abelian subgroup $\mathcal{L}$ of a real vector space $V$ is called a lattice if $\mathcal{L}=\mathbf{Z} v_1+\cdots+\mathbf{Z} v_r$ for some linearly independent vectors $v_1, \ldots, v_r$, and we say $\mathcal{L}$ is a full lattice if $r$ is the dimension of $V$ over $\mathbb{R}$. If $\mathcal{L}$ is such a full lattice with basis $v_1, \ldots, v_n$, then the set
$$
T=\left\{r_1 v_1+\cdots+r_n v_n: 0 \leq r_i<1, \quad 1 \leq i \leq n\right\}
$$
is called a fundamental parallelepiped for $\mathcal{L}$. From the definition of the map $v$ we know that $v(\mathcal{O}_K)$ is a full lattice in $\mathbb{R}^2$ with basis $v(1)=(1,1)$ and $v(\beta)=(\beta,-\dfrac{1}{\beta})$, and the fundamental parallelepiped of $v(\mathcal{O}_K)$ is shown in Figure \ref{lattice}. Further information regarding the lattice theory of algebraic number rings can be found in \cite{Janusz}.\\
Consider the triangle formed by the points $(0,0),(1,1)$ and $(\beta,\dfrac{1}{\beta})$. By calculating the coordinates of the circumcenter of this triangle, we find that it is located at $(\dfrac{3}{10}+\dfrac{2}{5}\beta,\dfrac{7}{10}-\dfrac{2}{5}\beta)$, while the radius of the circumcircle is $\sqrt{\frac{9}{10}}$, which is less than 1.\\
Let $\xi$ be any real quartic irrational belonging to a quadratic extension of $K$ whose conjugates are all real. Assume the minimal polynomial of $\xi$ is $(Ax^2+Bx+C)\cdot(\sigma(A)x^2+\sigma(B)x+\sigma(C))=0$. We arbitrarily select a root of $\sigma(A)x^2+\sigma(B)x+\sigma(C)=0$ and denote it by $\xi^\prime$. Define the vector $\boldsymbol{\Xi}=(\xi,\xi^\prime)$, one can express it as a sum of a lattice vector and a vector in the fundamental parallelepiped in polynomial time. Additionally, the latter vector can be modified to have length is less than $\sqrt{\frac{9}{10}}$. The details regarding this part can be found in the algorithm below. In other words, we can express $\boldsymbol{\Xi}=(a_0,\sigma(a_0))+\boldsymbol{\Xi^\prime}$, where $(a_0,\sigma(a_0))\in v(\mathcal{O}_K )$ and $|\boldsymbol{\Xi}|<\sqrt{\frac{9}{10}}<1$.\\
Based on the above discussion, for such a given $\xi$ and $\xi^\prime$, we define the continued fraction expansion algorithm for $\xi$:
\begin{algo}\label{alg2} The continued fraction expansion for pair $(\xi,\xi^\prime)$:
    \begin{enumerate}
    \item $\xi_0:=\xi$, $\xi^\prime_0:=\xi^\prime$, $\boldsymbol{\Xi_0}:=(\xi_0,\xi^\prime_0)$.
    
    \item For all $n\geq0$, find $\boldsymbol{a_n}:=(a_n,\sigma(a_n))\in v(\mathcal{O}_K)$ such that $|\boldsymbol{\Xi_n}-\boldsymbol{a_n}|<\sqrt{\frac{9}{10}}$. This can be achieved by performing the following calculation:
    \begin{enumerate}
        \item Solve the equation 
        \begin{equation*}
\begin{bmatrix}
1 & \beta \\
1 & -\frac{1}{\beta} \\
\end{bmatrix}
\begin{bmatrix}
\Tilde{x_n} \\
\Tilde{y_n} \\
\end{bmatrix}
=
\begin{bmatrix}
\xi_n \\
\xi_n^\prime \\
\end{bmatrix}
\end{equation*} 
\item Choose $x_n,y_n\in\mathbb{Z}$ such that $0<|x_n-\Tilde{x_n}|<1,\ 0<|y_n-\Tilde{y_n}|<1$ and $|x_n+y_n\beta-\xi_n|^2+|x_n-y_n\frac{1}{\beta}-\xi_n^\prime|^2<\frac{9}{10}$. One can find a satisfying $(x_n,y_n)$ after at most four attempts. Set $a_n=x_n+y_n\beta$.
    \end{enumerate}
    Note that the choices may not be unique one can choose one arbitrarily. This process can be completed in polynomial time.
    
    \item $\xi_{n+1}:=\dfrac{1}{\xi_n-a_n}$, $\xi^\prime_{n+1}:=\dfrac{1}{\xi^\prime_{n+1}-\sigma(a_n)}$,  $\boldsymbol{\Xi_{n+1}}:=(\xi_{n+1},\xi^\prime_{n+1})$.
\end{enumerate}
\end{algo}

Then we obtain an infinite continued fraction expansion of $\xi$ by using the aforementioned algorithm, which generates the sequence $[a_0,a_1,\cdots]$, with all partial quotients belong to $\mathcal{O}_K$. It can be seen that both $|\xi_n|$ and $|\xi_n^\prime|$ are larger than $\sqrt{\frac{10}{9}}$.\\
The following lemma is derived from Corollary 4.5 discussed in \cite{dani}, which pertains to the properties of complex continued fractions with the proof remains essentially unchanged.
\begin{lem}\label{lemeg}
    Suppose $\xi$ as above is a real quartic irrational contained in some quadratic extension of K whose conjugate roots are all real, and the continued fraction expansion of $\xi$ over K is given by our defined algorithm with respect to a chosen conjugate root $\xi^\prime$. Apply Algorithm \ref{alg2} on the pair $(\xi,\xi^\prime)$. Let 
 $\gamma=\sqrt{\frac{10}{9}}$. Denote $N:=\{n\in\mathbb{N}\ |\ |Q_{n}|>|Q_{n-1}|\}$. If $N$ is not empty, say the minimal element in $N$ is $n_0$, then for all $n\geq n_0$, we have $$|Q_n\xi-P_n|<\dfrac{(\gamma-1)^{-1}}{|Q_n|}.$$
 Similarly for $\xi^\prime$, we can define $N^\prime:=\{n\in\mathbb{N}\ |\ |\sigma(Q_{n})|>|\sigma(Q_{n-1})|\}$. If $N^\prime$ is not empty with minimal element $n_0^\prime$, then for all $n\geq n_0^\prime$, we have 
 $$|\sigma(Q_n)\xi^\prime-\sigma(P_n)|<\dfrac{(\gamma-1)^{-1}}{|\sigma(Q_n)|}.$$
\end{lem}
\begin{proof}
 We prove the claim only for $\xi$ since the proof for $\xi^\prime$ is essentially identical.\\
 By the definition of our algorithm, we have $|\xi_n|> \gamma>1$ for all $n$. Thus if $n_0\in N\neq \emptyset$, then $|\xi_{n_0+1}+\dfrac{Q_{n_0}}{Q_{n_0-1}}|>|\xi_{n_0+1}|-|\dfrac{Q_{n_0}}{Q_{n_0-1}}|>(\gamma-1)^{-1}$, which implies that $|Q_{n_0}\xi-P_{n_0}|=|\dfrac{1}{\xi_{n_0+1}Q_{n_0}+Q_{n_0-1}}|<\dfrac{(\gamma-1)^{-1}}{|Q_{n_0}|}$. If $n_0+1\in N$, we can similarly deduce that $|Q_{n_0+1}\xi-P_{n_0+1}|<\dfrac{(\gamma-1)^{-1}}{|Q_{n_0+1}|}$. Otherwise if $n_0+1\notin N$, then $|Q_{n_0+1} |\leq|Q_{n_0}|$. Hence $|Q_{n_0+1}\xi-P_{n_0+1} |=|\xi_1\cdots\xi_{n_0+2}|^{-1}<|Q_{n_0}\xi-P_{n_0}|<\dfrac{(\gamma-1)^{-1}}{|Q_{n_0}|}\leq\dfrac{(\gamma-1)^{-1}}{|Q_{{n_0}+1}|}$. Therefore, by induction, we can conclude that for all $n\geq n_0$ we have $|Q_{n}\xi-P_{n}|<\dfrac{(\gamma-1)^{-1}}{|Q_{n}|}$.\\
\end{proof}
Now we demonstrate that this algorithm will generate an ultimately periodic continued fraction expansion that converges to $\xi$.
\begin{thm}\label{thmeg}
With all the assumptions in Lemma \ref{lemeg}, the continued fraction obtained by $(\xi,\xi^\prime)$ using Algorithm \ref{alg2} is ultimately periodic. Also, if we denote the expansion by $P$, then $P$ takes $\xi$ as the limit, while $\sigma(P)$ converges to the chosen $\xi^\prime$.
\end{thm}
\begin{proof}
   If both of the sets $N$ and $N^\prime$, as defined  above, are not empty, then Lemma \ref{lemeg} holds. From Proposition \ref{prop1} we have $|\xi Q_n-P_n|=|\xi_1\cdots\xi_{n+1}|^{-1}$, $|\xi^\prime\sigma(Q_n)-\sigma(P_n)|=|\xi_1^\prime\cdots\xi_{n+1}^\prime|^{-1}$, both tending to 0 as $n$ increases, given that $|\xi_i|,|\xi_i^\prime|>\sqrt{\frac{10}{9}}$. Thus according to Theorem \ref{thm>0}, the expansion is ultimately periodic. \\
   On the other hand, if $N$ is empty, then $|Q_n|$ is bounded, hence $|P_n|$ is also bounded, as $|Q_n\xi-P_n|=|\xi_1\cdots\xi_n|^{-1}\rightarrow0$. The similar result holds if $N^\prime$ is empty. Therefore, we always have bounded $|A_n|=|AP_n^2+BP_n+C_n|$ and $|\sigma(A_n)|=|\sigma(AP_n^2+BP_n+C_n)|$, and by Theorem \ref{thmbound} we conclude again that expansion is ultimately periodic.\\
Denote the obtained continued fraction expansion by $P$, then Lemma \ref{lem!} guarantees the convergence of $P$ and $\sigma(P)$, and by Lemma \ref{lem@}, both $|Q_n|,|\sigma(Q_n)|\to\infty$, hence by $|\xi Q_n-P_n|,\ |\xi^\prime\sigma(Q_n)-\sigma(P_n)|\to0$, we can establish the convergence of $P$ to $\xi$ and $\sigma(P)$ to $\xi^\prime$, as required in the theorem.
\end{proof}
\begin{rmk}\label{rmkeg}
   Although we have discussed the cases whether the set $N(N^\prime)$ is empty during the proof, it should be seen by Lemma \ref{lem@} neither of them can be empty.\\
   Our approach depends heavily on the property of the lattice $v(\mathcal{O}_K)$, considering the definition of the covering radius as the minimum radius that encompasses the closed circles centered at vertices covering the polygon. In fact, our method is only applicable to those $v(\mathcal{O}_K)$ with covering radius less than 1.\\
   Using the method in \cite{yupeng2012covering}, we can easily determine the covering radius $r$ of $v(\mathcal{O}_K)$
    for $K=\mathbb{Q}(\sqrt{D})$. When $D\equiv2,3\mod4$, $r=\sqrt{\dfrac{D+1}{2}}$, and for $D\equiv1\mod 4$, $r=\dfrac{1}{2\sqrt{2}}(\sqrt{D}+\dfrac{1}{\sqrt{D}})$. One can check that for $D\neq 5$ the covering radius is always larger than 1.
\end{rmk}
\begin{eg}
Let $ \xi = 1 + \sqrt{\beta^2 +1}$ and $\xi^\prime = 1 + \sqrt{\frac{1}{\beta^2}+ 1}$. Define $\boldsymbol{\Xi_0} = (\xi,\xi^\prime)$. Let $a_0 = 2$, and $\boldsymbol{\Xi_1} = (\dfrac{\sqrt{\beta^2+1}}{\beta^2},\beta^2+\sqrt{\beta^4+\beta^2})$. If we take $a_1=4-2\beta$, then $\boldsymbol{\Xi_2}=\boldsymbol{\Xi_0}$, hence we obtain a purely periodic continued fraction expansion of $\xi=\overline{[2,4-2\beta]}$ with period length 2. Alternatively, take $\xi^\prime=1-\sqrt{1+\frac{1}{\beta^2}}$, we will obtain another expansion of $\xi=[\beta^2,\overline{2\beta}]$ that is ultimately periodic with period length 1.
\end{eg}

\bibliographystyle{IEEEtran}
\bibliography{ref}

\end{document}